\documentclass[10pt,reqno]{amsart}
\usepackage{amsmath}
\usepackage{amsfonts}
\usepackage{amssymb}
\usepackage{amsthm}
\usepackage{amscd}
\usepackage{a4wide}
\usepackage{enumerate}
\usepackage{dsfont} 

\usepackage{graphics}
\usepackage{eucal}
\usepackage{mathrsfs}

\usepackage{xcolor}
\usepackage[pdftex,bookmarks,colorlinks,breaklinks]{hyperref}  

\hypersetup{linkcolor=red,citecolor=blue,filecolor=dullmagenta,urlcolor=darkblue} 

\newtheorem {theorem}    {Theorem}[section]
\newtheorem {lemma}      [theorem]    {Lemma}

\newtheorem {corollary}  [theorem]    {Corollary}
\newtheorem {proposition}[theorem]    {Proposition}

\theoremstyle{definition}
\newtheorem {definition} [theorem]    {Definition}
\newtheorem {example}    [theorem]    {Example}
\newtheorem {remark}    [theorem]    {Remark}

\newcounter{AbcT}

\numberwithin{equation}{section}

\newcommand{\IGNORE}[1]{}

\newcommand{\bigzero}{\mbox{\normalfont\large\bf 0}}

\DeclareMathOperator{\GL}{GL}

\newcommand{\frk}[2][n]{\mathfrak{#2}}

\newcommand{\defeq}{\stackrel{\text{def}}{=}}

\begin{document}
\title{Dirichlet improvability for $S$-numbers}

\begin{abstract}
We study the problem of improving Dirichlet's theorem  of metric Diophantine approximation in the most general and multiplicative form in $S$-adic setting. Our approach is based on translation of the problem related to Dirichlet improvability into a dynamical one, and the main technique of our proof is the $S$-adic version of quantitative nondivergence estimate due to D.Y. Kleinbock and G. Tomanov. The main result of this paper can be regarded as the number field version of earlier works of D.Y. Kleinbock and B. Weiss \cite{KW1}, and of the second named author and Anish Ghosh \cite{GG}. Also this in turn generalises a result of Shreyasi Datta and M. M. Radhika \cite{DR} on singularity of vectors to any number field $K$ and $S$ containing all archimedian places. 
\end{abstract}
\subjclass[2020]{11J83, 11K60, 37A17, 11J13} \keywords{
$S$-adic Diophantine approximation, Dirichlet's theorem, Dirichlet improvability}
\thanks{}

\author{Sourav Das}
\author{Arijit Ganguly}

\address{Department of Mathematics and Statistics, Indian Institute of Technology Kanpur, Kanpur, 208016, India}
\email{iamsouravdas1@gmail.com, arijit.ganguly1@gmail.com}
\maketitle

\section{Introduction}\label{section:intro}
In this paper we study the improvement of Dirichlet's theorem of Diophantine approximation in the $S$-adic setting. The branch of Number theory called `Diophantine approximation' deals with approximation of real numbers by rationals and its higher dimensional analogues, and the Dirichlet's theorem is one of its foundational results. It is indeed quite natural to ask how much improvement the  above mentioned theorem of Dirichlet undergoes for a given real number, vector, matrix etc. Recall that given $A \in M_{m \times n}(\mathbb R)$ and $ 0 < \varepsilon < 1,$ we say that Dirichlet's theorem can be $\varepsilon$-improved for $A$, and write $A \in DI_{\varepsilon}(m,n)$, if for any $t \gg 1$, there exists $\mathbf{q} \in \mathbb Z^n \setminus \{0\}$ and $\mathbf{p} \in \mathbb Z^m$ satisfying 
\[ \|A \mathbf{q} - \mathbf{p} \|\, < \, \varepsilon e^{-t/m}    \text{ \,\,and } \displaystyle 
\|\mathbf{q} \| < \varepsilon e^{t/n},\]  where $ \| \cdot \|$ denotes the sup norm on $ \mathbb R^n$.  Also we say that $A \in M_{m \times n}(\mathbb R)$ is $\textit{singular}$ if $A \in DI_{\varepsilon}(m,n)$ for any $\varepsilon >0$. In the pioneering works \cite{DS1,DS2}, H. Davenport and W. Schmidt showed that the irrational numbers that are Dirichlet improvable are precisely the ones which are badly approximable (\cite{DS1}). Furthermore, they have also established the first metrical result  in the complete generality  which says that,
$DI_{\varepsilon}(m,n)$ has zero Lebesgue measure for any $\varepsilon\in (0,1)$.\\

In 1964, V. Sprind\v{z}uk developed a new avenue, usually referred to as `Metric Diophantine approximation on manifolds', while resolving a long standing conjecture of K. Mahler that says,  
 almost every point on the curve $(x,x^2,\dots,x^n)$ is not `very well approximable' by rationals with respect to the induced one dimensional Lebesgue measure. The central problem of this subject is to investigate the extent the properties of a generic point in $\mathbb{R}^n$ with respect to Lebesgue measure or some nice measures are inherited by embedded submanifolds. Later on V. Sprind\v{z}uk himself conjectured a more general form of that of Mahler.  Several partial progress on this conjecture had been made, yet the original conjecture seemed to remain beyond one's reach. In 1998, D.Y. Kleinbock and G.A. Margulis proved the above mentioned conjecture of Sprind\v{z}uk in their landmark work \cite{KM} using technique from homogeneous dynamics. Undoubtedly it was a stroke of genius for them to observe a Dani type correspondence between \emph{very well approximability} of a vector and cusp excursion behaviour of the trajectory of certain unimodular lattice. At the core of their argument there is a sharp quantitative estimate, known as the `Quantitative nondivergence' estimate,  on nondivergence of trajectories under multidimensional diagonal flows that shows trajectories with such a nature are rare. Soon this Quantitative nondivergence estimate has been generalized, extended to many directions. To mention a few:
\begin{itemize}
 \item In \cite{KLW}, D.Y. Kleinbock, E. Lindenstrauss and B.
Weiss  generalize the results of \cite{KM} to more general class of measures, namely `friendly measures'.
\item D.Y. Kleinbock and G. Tomanov established the quantitative nondivergence estimate for flows on homogeneous spaces of products of real and $p$-adic Lie groups in \cite{KT}. Using this, they proved the $S$-adic analogue of \cite{KM}.
\end{itemize}
The reader is also suggested to see \cite{BKM,K,G-JNT} for other generalizations and extensions of that estimate which have number theoretic implications.  \\

The question of  Dirichlet improvability for points lying in manifolds has been first studied by Davenport and Schmidt, in fact they  showed in \cite{DS1} that, the set of all $x \in \mathbb R$ for which $(x,x^2) \in DI_{\varepsilon}(1,2)$ has zero Lebesgue measure for any $\varepsilon < 4^{-1/3}.$ Further developments in this direction were made by R.C. Baker \cite{Ba1,Ba2}, Y. Bugeaud \cite{Bu}, and M. Dodson, B. Rynne and J. Vickers \cite{DRV,DRV1}. In \cite{KW1} Kleinbock and  Weiss significantly generalize earlier works of Davenport, Schmidt, Baker and Bugeaud on Dirichlet improvability. One of the main results of \cite{KW1} is the existence of a constant $\varepsilon_0 >0$ such that for all $ \varepsilon < \varepsilon_0,$ $\mathbf{f}_* \nu(DI_{\varepsilon}(1,n))=0,$ for continuous, \emph{good} and \emph{nonplanar} maps $\mathbf{f}$ and Radon, \emph{Federer} measure $\nu.$ Note that the above mentioned $\varepsilon_0$ is quite far from $1$ and getting the same result as above for all $\varepsilon < 1$ is open and seems to be quite challenging. For analytic curves not contained in any affine hyperplane, this has been resolved by N. Shah in \cite{Shah1}. See also \cite{Shah2, Shah3} for some more results related to this. \\

Sooner or later,  there has been a drive to generalize classical results in Diophantine approximation to the setting of an arbitrary number field. The adelic successive minima theorem has been established independently by McFeat \cite{MCF}  and Bombieri and Vaaler \cite{BomVa}. See also \cite{B, KST},  and the references therein for  $S$-adic version of other foundational results of geometry of  numbers. There seem to exist multiple versions, although with very little variance, of Dirichlet's theorem over number fields in the literature (\cite{B,H,Q,Sch1,Sch2}). In the recent work \cite{AlG}, Mahbub Alam and Anish Ghosh obtained several quantitative improvements to Dirichlet's theorem in number fields generalizing an weighted spiralling of approximates.  The metric theory on manifolds in this context has germinated some time ago. The $S$-arithmetic version of Sprind\v{z}uk's conjecture has been proved by Kleinbock and Tomanov in \cite{KT}  modifying the technique of \cite{KM}. Subsequently using the results of \cite{KT}, A. Mohammadi and A. Salehi Golsefidy have settled the convergence case of $S$-arithmetic Khintchine-type theorem for $S$-adic non-degenerate analytic manifolds in \cite{MS1} and \cite{MS2}. An inhomogeneous theory on manifolds of the same,  both convergence and divergence case, is developed in the S-arithmetic setting in \cite{DG1} by Shreyasi Datta and Anish Ghosh. In a later work \cite{DG2}, the authors have also proved the $p$-adic analogue of inheritance of Diophantine exponents for affine subspaces, in a stronger form indeed. The badly approximable vectors have also been studied extensively in number fields in several works (see \cite{EGL, KLy, AGGLy}).\\

Since together with the euclidean and $S$-adic ones function field counterparts complete the general theory of Diophantine approximation over local fields, it is worth mentioning some of 
developments that took place of late in the setting of function fields. The reader should note that there are many interesting parallels, i.e., many results of usual euclidean Diophantine approximation also hold over function fields, however there are certain striking exceptions as well. For example, there is no analogue of Roth's theorem for function fields over a finite field. The geometry of numbers in the context of function fields was developed by K. Mahler (\cite{M}) in $1940$, and using results analogous to that of the euclidean case, one can obtain Dirichlet type theorem in positive characteristic (see \cite{GG} for an elementary proof of the most general form of Dirichlet's theorem in this setting). V. Sprind\v{z}uk established the positive characteristic analogue of Mahler's conjecture in \cite{Spr1}. We refer the reader to \cite{deM,L1} for general surveys and to \cite{AGP,G-JNT,GG1,GR,KN,Kr1,Kr2,L2} for some of the recent developments. Sprind\v{z}uk's conjecture over a local field of positive characteristic was settled by Anish Ghosh in \cite{G-pos}. The improvability of  Dirichlet's theorem has been first studied in this context by Arijit Ganguly and Anish Ghosh \cite{GG, GG1}. Indeed, the main result of \cite{GG} can be regarded as the function field analogue of that of one of the main results  of \cite{KW1}. On the other hand, much to one's surprise, it has been shown in \cite[Theorem 2.4]{GG1} that the Laurent series that are Dirichlet improvable are precisely the rational functions, which is completely opposite to the euclidean scenario. \\ 

It is quite expected to be inquisitive about Dirichlet improvability over number fields. Notwithstanding that in this setting there have been multiple versions of Dirichlet's theorem in the literature as mentioned earlier,  the problem of improvability of the same does not seem to have been pursued much to the best of author's knowledge; although a more restrictive notion, namely \emph{singular vectors}, for totally real number fields $K$ and $S$ being the collection of all archimedian places, has been coined and discussed at length in the very recent paper \cite{DR} by Shreyasi Datta and M. M. Radhika. We take up this as the theme of this project.
The main result of this paper, namely Theorem \ref{thm:1.4}, similar to those of \cite[Theorem 1.5]{KW1} and \cite[Theorem 3.7]{GG}, deals with the Dirichlet improvability of a generic point, with respect to some nice class of measures, lying in the image of a good and nonplanar map in the $S$-adic setting. Thus Theorem \ref{thm:1.4} of this paper can be regarded as the number field version of \cite[Theorem 1.5]{KW1} and \cite[Theorem 3.7]{GG}. \\

Further to the above, the notion singularity of vectors is now legitimately extended to any arbitrary number field $K$ and $S$ containing all archimedian places, and that, too, in any direction (see Definition \ref{defn:sing1}). This is indeed equivalent to that of \cite{DR} (see Lemma \ref{lemma:equiv}), when $K$ is totally real and $S$ consists of all archimedian places. From whence our Theorem \ref{thm:1.4} can be regarded as a generalization of \cite[Theorem 2.3]{DR}.\\


We first establish a multiplicative version (see Theorem \ref{thm:1.1}) of  Dirichlet's theorem for $S$-adic matrices. Although it can be proved in the exactly similar manner to that of \cite[Lemma 5.1]{B}, it is clearly the most general version of Dirichlet's theorem over number fields. In order to provide  a brief  overview of our result to the reader, here we first state a simplified version of the above mentioned Theorem \ref{thm:1.1} as follows. To begin with, we suppose $K$ is a number field of degree $d$ over $\mathbb Q.$ The completion of $K$ with respect to the place $v$ is denoted by $K_v$. Let $S$ be a finite collection of pairwise nonequivalent valuations of $K$ containing all archimedean places. By $K_S$ we denote the direct product of all $K_v$ for $v \in S$. Let $\mathcal O_S = \{x \in K : |x|_v \leq 1 \,\, \forall \,\, v\notin S \}$ denotes the $\textit{ring of S-integers}$ of $K,$ where $|\cdot|_v$ denotes the normalized absolute value in $K_v.$ We define the field constant 
    \begin{equation}
        \text{const}_K=\left(\frac{2}{\pi}\right)^s|D_K|^{1/2},
    \end{equation}
  where $s$ is the number of complex places of $K$ and $D_K$ is the discriminant of $K.$
  
\begin{theorem}
  For each $v \in S,$ let $\mathbf{y}^{(v)}=(y_1^{(v)},\dots,y_n^{(v)}) \in K_v^n.$ Also let $\varepsilon_v,$ $\delta_v$ $\in K_v \setminus \{0\}$ be given for each $v \in S$ with $|\varepsilon_v|_v <1$ and $|\delta_v|_v \geq 1$ for each $v \in S$ and
  $$\prod_{v \in S} |\varepsilon_v|_v |\delta_v|_v^n=(\text{const}_K)^{n+1}.$$
  Then there exist $ \vec{\mathbf{q}}=(q_1,\dots,q_n) \in \mathcal O_S^n \setminus \{0\} $ and $p \in \mathcal O_S$ satisfying
  \[| q_1 y_1^{(v)}+q_2 y_2^{(v)}+ \dots +q_n y_n^{(v)} - p|_v \, \leq \, |\varepsilon_v|_v \text{ \,\,and } \displaystyle 
\max_{1\leq j \leq n} |q_j|_v \leq |\delta_v|_v \,,\]
for all $v \in S.$
\end{theorem}
\noindent For proof, we refer to  \cite[Lemma 5.1]{B}. Now the notion of Dirichlet improvability can be defined as follows:
 \begin{definition}\label{defn:dimp0}
 Given a vector $\mathbf{y}=\left(\mathbf{y}^{(v)}\right)_v \in K_{S}^{n},$ where  $\mathbf{y}^{(v)}=(y_1^{(v)},\dots,y_n^{(v)}) \in K_v^n$ for each $v\in S$ and $0<\varepsilon<1.$ We say that Dirichlet's theorem can be $\varepsilon$-improved for $\mathbf{y}$ or use the notation $\mathbf{y} \in DI_{\varepsilon}(1,n,K_S)$ if given any $\varepsilon_v,$ $\delta_v$ $\in K_v \setminus \{0\}$  for each $v \in S$ with $|\varepsilon_v|_v <1$ and $|\delta_v|_v \geq M$ for each $v \in S,$ where $M \geq 1$ is sufficiently large and
  $$\prod_{v \in S} |\varepsilon_v|_v |\delta_v|_v^n=(\text{const}_K)^{n+1},$$
  there exist $ \vec{\mathbf{q}}=(q_1,\dots,q_n) \in \mathcal O_S^n \setminus \{0\} $ and $p \in \mathcal O_S$ satisfying
  \[| q_1 y_1^{(v)}+q_2 y_2^{(v)}+ \dots +q_n y_n^{(v)} - p|_v \, \leq \, \varepsilon |\varepsilon_v|_v \text{ \,\,and } \displaystyle 
\max_{1\leq j \leq n} |q_j|_v \leq \varepsilon |\delta_v|_v \,,\]
for all $v \in S.$
 \end{definition}

In the next couple of sections we will prove multiplicative version of Dirichlet's theorem over number fields (Theorem \ref{thm:1.1})  and define a more general version of Dirichlet improvability (Definition \ref{defn:DI imp}). Now we state a special case of our  Theorem \ref{thm:1.4}.
\begin{theorem}
Let $X= \prod_{v \in S} K_{v}^{l_v},$ where $l_v \in \mathbb N $  and $U=\prod_{v \in S} U_v \subseteq X$ be open, and let $ \mathbf{f}=\left(\mathbf{f}^{(v)}\right)_v:U \rightarrow K_S^n$  be a map, where   $\mathbf{f}^{(v)}=\left(f_1^{(v)},\dots,f_n^{(v)}\right)$ are continuous maps from $U_v$ to $K_v^n$ such that $f_1^{(v)},\dots,f_n^{(v)}$ are polynomials and $1,f_1^{(v)},\dots,f_n^{(v)}$ are linearly independent over $K_v$ for each $v\in S.$ Then there exists  $\varepsilon_0 >0$ such that for all $\varepsilon < \varepsilon_0,$ $\mathbf{f}(\mathbf{x})$ is not Dirichlet $\varepsilon$-improvable for $\lambda$ almost every $\mathbf{x} \in U,$ where $\lambda =\prod_{v \in S} \lambda_v$ such that $\lambda_v$ is a Haar measure on $K_v^{l_v}.$ 
\end{theorem}
Note that one can obtain an explicit estimate of the constant $\varepsilon_0$ appearing in the above theorem; indeed it is not difficult to compute that in terms of known quantities for $n=2,3$ and nice maps etc. So in that sense our proof can be regarded as `effective'. However optimality for that value of $\varepsilon_0$, computed as above, cannot be expected; although at the same time the set up of  Theorem \ref{thm:1.4} is very much general. Our proof is based on the $S$-arithmetic version of the quantitative nondivergence estimate (see \S6.6 \cite{KT-pre}).

\subsection*{Acknowledgements} The authors sincerely thank Anish Ghosh, School of Mathematics, Tata Institute of Fundamental Research, Mumbai, for all his constant supports. We are indebted to the Department of Mathematics and Statistics, Indian Institute of Technology Kanpur, for providing the research friendly atmosphere. We also want to express our gratitude to Shreyasi Datta for her valuable comments.  
 
\section{Dirichlet's theorem over number fields}\label{section:Dirichlet type }

In this section, we first state  Dirichlet's  theorem over number fields, which is a slight modification of  \cite[Lemma 5.1]{B}. Before stating Dirichlet's theorem, we recall some notations and terminologies from \cite{B}. 

 We take $K$ to be an algebraic number field of degree $d$ over $\mathbb Q.$ The collection of all nontrivial places of $K$ is denoted by $P_K$. For any $v \in P_K$, $K_v$ will denote the completion of $K$ with respect to the place $v$. We know that for any $v \in P_K$, $K_v$ is isomorphic to a finite extension of either $\mathbb R$ or $\mathbb Q_p$ for some prime $p$. Now for any given $v \in P_K$, if $v $ is an archimedean place, we say that $v$ lies over infinity, denoted by $v|\infty$. If $v$ is a nonarchimedean place then $v$ is the extension of some $p$-adic valuation. In this case we say $v$ lies over finite prime $p$, denoted by $v|p$. 
 \begin{itemize}
     \item $d_v=[K_v : \mathbb Q_v]$ denotes the local degree of $K_v$ for each place $v \in P_K$.
     \item $\|\cdot\|_v$ denotes the absolute value on $K_v$ for each place $v \in P_K$.
 \end{itemize}
 
For each place $v\in P_K$, we normalize the absolute value $\|\cdot\|_v$ as follows:
\begin{enumerate}
    \item if $v|p$ then $\|p\|_v=p^{-1},$
    \item if $v|\infty$ then for $x\in K_v,$ $\|x\|_v=|x|$, where $|\cdot|$ is the Euclidean absolute value on $\mathbb R$ or $\mathbb C$.
\end{enumerate}
\noindent So $\|\cdot\|_v$ extends the usual $p$-adic absolute value if $v|p$ and the Euclidean absolute value if $v|\infty$. Our second normalized absolute value $|\cdot|_v$ on $K_v$ is defined by
$$|x|_v=\|x\|_{v}^{d_v}.$$
Due to this normalization, we have the product formula:
\begin{equation}\label{equ:1.1}
 \prod_{v \in P_K} |x|_v =1   
\end{equation}
for all $x \in K,$ $x\neq 0.$

\begin{itemize}
    \item For $\vec{x}=(x_1,x_2,\dots,x_n) \in K_{v}^{n}$, we extend our absolute values as follows:
    $$|\vec{x}|_v=\max_{1\leq i \leq n} \{|x_i|_v\},$$
    $$\|\vec{x}\|_v=\max_{1\leq i \leq n} \{\|x_i\|_v\}.$$
    \item Let $v$ be a finite place of $K.$ We denote the maximal compact (open) subring of $K_v$ by $\mathcal O_v,$ that is,
    $$\mathcal O_v = \{x \in K_v : |x|_v \leq 1 \}.$$
    \item A subset $R_v$ in $K_v^n$ is a $K_v$-lattice if it is a compact open $\mathcal O_v$-module in $K_v^n.$ It is easy to observe that, $\mathcal O_v^n$ is a $K_v$-lattice in $K_v^n.$
    \item Let $S$ be a finite collection of places of $K$ containing all archimedean places. We define the ring of $S$-integers as 
    $$\mathcal O_S =\{x \in K : x \in \mathcal O_v \,\, \text{for} \,\, \text{all} \,\, v \notin S\}.$$
    \item The adele ring of $K$ is denoted by $K_{\mathbb A}$. $K$ is diagonally embedded inside $K_{\mathbb A}$, so that we may view $K \subseteq K_{\mathbb A}$ by the natural diagonal map $\rho : K \rightarrow K_{\mathbb A}$ defined by 
    $$\rho(\alpha)=(\alpha,\alpha,\alpha,\dots )$$
    for $\alpha \in K.$
    \item We say that a nonempty subset $R_v \subseteq K_v^n$ a regular set if it has one of the following forms:
    \begin{enumerate}
        \item If $v|\infty$ then $R_v$ is a bounded, convex, closed, symmetric subset of $K_v^n$ with nonzero volume.
        \item If $v \nmid \infty$ then $R_v$ is a $K_v$-lattice in $K_v^n$.
        
    \end{enumerate}
    
    \item We say that a subset $\mathcal R $ of $K_{\mathbb A}^n$ is \textit{admissible} if $\mathcal{R}= \prod_{v \in P_K} R_v,$ where $R_v$ is a regular set in $K_v^n$ for each $v \in P_K$ and $R_v=\mathcal{O}_v^n$ for almost all places $v.$
    
    \item We define the field constant for the number field $K$ as:
    \begin{equation}\label{equ:1.2}
        \text{const}_K=\left(\frac{2}{\pi}\right)^s|D_K|^{1/2},
    \end{equation}
    \noindent where  $D_K$ is the discriminant of $K$ and $s$ is the number of complex places of $K.$
    
\end{itemize}

\noindent We are now ready to state and prove a version of Dirichlet's theorem over number fields. The following theorem is a slight modification of  \cite[Lemma 5.1]{B}.

\begin{theorem}\label{thm:1.1} Let $A_v$ be an $m \times n$ matrix over $K_v$ for each $v \in S.$ Also let $$\varepsilon_{v}^{(1)},\dots,\varepsilon_{v}^{(m)},\delta_{v}^{(1)},\dots,\delta_{v}^{(n)} \in K_v \setminus \{0\}$$ be given for each $v \in S$ so that $| \varepsilon_{v}^{(i)} |_v < 1, \,\, \forall \,\, i=1, \dots,m$ and $|\delta_{v}^{(j)}|_v \geq 1, \,\, \forall  \,\, j=1,\dots,n$ for each $v \in S$ and 
\begin{equation}\label{equ:1.3}
    \prod_{v \in S}| \varepsilon_{v}^{(1)} |_v \dots | \varepsilon_{v}^{(m)} |_v |\delta_{v}^{(1)}|_v \dots |\delta_{v}^{(n)}|_v=(\text{const}_K)^{m+n}.
\end{equation}
Then there exists $\vec{x}=(x_1,\dots,x_n) \in \mathcal O_S^n \setminus \{\vec{0}\}$ and $\vec{y}=(y_1,\dots,y_m) \in \mathcal O_S^m$ satisfying
\begin{equation}\label{eqn:1.4}\left \{ \begin{array}{rcl} \|A_{v}^{(i)} \vec{x} - y_i \|_v  \leq \|\varepsilon_{v}^{(i)}\|_v &\mbox{for} &i=1,2,...,m \\ \|x_j\|_v \leq \|\delta_{v}^{(j)}\|_v &\mbox{for} &j=1,2,...,n\,, \end{array} \right.  
\end{equation}
for all $v \in S,$ where $A_{v}^{(i)}$ is the ith row of $A_v$ for $i=1,2,\dots , m.$
\end{theorem}
\begin{proof}
The proof of this theorem goes along the lines of the original proof by Burger with slight modification. For each place $v \in P_K,$ we define a $(m+n)\times (m+n)$ matrix $B_v$ over $K_v$ by:
 \[
B_v=\begin{pmatrix}
  \begin{matrix}
   {\delta_v}^{-1}I_n
  \end{matrix}
   & \bigzero \\

  \varepsilon_v^{-1}A_v &
  \begin{matrix}
  {\varepsilon_v}^{-1}I_m
  \end{matrix}
\end{pmatrix},
\]
for all $v \in S,$ where $\varepsilon_v^{-1}A_v$ is the $m \times n$ matrix over $K_v,$ whose $i$th row is ${\varepsilon_v^{(i)}}^{-1}A_v^{(i)}$ for $i=1, \dots,m;$  ${\delta_v}^{-1}I_n=\text{diag}\left({\delta_v^{(1)}}^{-1}, {\delta_v^{(2)}}^{-1}, \dots, {\delta_v^{(n)}}^{-1}\right)$ and  $ \varepsilon_v^{-1}I_m=\text{diag}\left({\varepsilon_v^{(1)}}^{-1}, {\varepsilon_v^{(2)}}^{-1},  \dots, {\varepsilon_v^{(m)}}^{-1}\right).$   
Also define $B_v=I_{m+n}$ for all $v \notin S.$

Define $R_v \subseteq K_{v}^{m+n}$ by 
$$R_v=\{\vec{z} \in K_{v}^{m+n} : \|B_v \vec{z}\|_v \leq 1 \}$$
and let $\mathcal R=\prod_{v} R_v.$ Now it is easy to see that $\mathcal R \subseteq K_{\mathbb A}^{m+n}$ and $\mathcal R $ is admissible. Observe that 

\begin{align*}
  \displaystyle  V(\mathcal R)& =2^{d(m+n)}\left(\frac{\pi}{2}\right)^{s(m+n)}|D_K|^{-(m+n)/2} \left(\prod_{v}|\text{det}B_v|_v\right)^{-1}  \\ & \displaystyle
     =2^{d(m+n)} \left(\left(\frac{\pi}{2}\right)^{s}|D_K|^{-1/2}\right)^{m+n} \left(\prod_{v} | \varepsilon_{v}^{(1)} |_v \dots | \varepsilon_{v}^{(m)} |_v |\delta_{v}^{(1)}|_v \dots |\delta_{v}^{(n)}|_v\right) \\ & \displaystyle =2^{d(m+n)}.
\end{align*}

Let $\lambda_1, \lambda_2, \dots, \lambda_{(m+n)}$   be the successive minima of $\mathcal R \subseteq K_{\mathbb A}^{m+n}$ with respect to $K^{m+n}.$ Then by the adelic successive minima theorem \cite[Theorem 3]{BomVa}, we know that
$$(\lambda_1 \lambda_2 \dots \lambda_{m+n})^d V(\mathcal R) \leq 2^{d(m+n)}.$$
Thus we have $\lambda_1 \leq 1,$ since $\lambda_n$'s are increasing and $V(\mathcal R )=2^{d(m+n)}.$ Hence there exists a point $\begin{pmatrix} \vec{x} \\ - \vec{y}
         \end{pmatrix}$ $\in K^{m+n}\setminus \{0\} \cap \mathcal R.$ By our definition of $R_v$ for $v \notin S,$ we have 
         $$\begin{pmatrix} \vec{x} \\ - \vec{y}
         \end{pmatrix} \in \mathcal{O}_{S}^{m+n}.$$
Also for $v \in S ,$ we have 
$$\|A_{v}^{(i)} \vec{x} - y_i \|_v  \leq \|\varepsilon_{v}^{(i)}\|_v$$
and $$\|x_j\|_v \leq \|\delta_{v}^{(j)}\|_v$$
for $i=1,2, \dots,m$ and $j=1,2,\dots,n.$ And it is easy to see that $\vec{x} \neq \vec{0}$ due to the product formula (\ref{equ:1.1}).
\end{proof}

\section{Dirichlet improvability and the main theorem}
\indent Before introducing the notion of ``Dirichlet improvability'', we first recall few notations from \cite{KT-pre}.

\begin{itemize}
    \item Throughout this article $K$ will represent a number field of degree $d$.
    \item $S$ will denote  a finite set of places of $K$ containing all the archimedean places. Precisely, let $S=\{v_1,\dots,v_{\ell}\},$ where $v_i \in P_K$ for $i=1,\dots,\ell,$ such that $S$ contains all archimedean places.
    \item We will denote the normalized absolute value on $K_v$ by $|\cdot|_v,$ which is same as earlier defined normalization $|\cdot|_v$ on $K_v.$ 
    \item $K_S$ is the direct product of all the completions $K_v,$ $v \in S,$ that is
    $$K_S=\prod_{v \in S}K_v.$$
    $K$ is imbedded inside $K_S$ via the diagonal imbedding. The set of all archimedean places of $K$ is denoted by $S_a \subseteq S$ and we denote $S \setminus S_a $ by $S_f.$ Let $S_c$ and $S_r$ respectively denote the set of all complex places and the set of all real places of $K.$
    
    \item The elements of $K_S$ will be denoted as $x=(x^{(v)})_{v \in S},$ or simply $x=(x^{(v)})_v,$ where $x^{(v)} \in K_v.$ For $x \in K_S$ we define the content $c(x)$ of $x$ as follows:
    $$c(x)=\prod_{v \in S}|x^{(v)}|_v .$$
    
    \item As usual the ring of integers of $K$ is denoted by $\mathcal O_K.$
    \item $\mathcal O_S$, i.e., the ring of $S$-integers of $K$, is also diagonally embedded inside $K_S.$ We will use this identification conveniently.
    \item  We will denote the elements of $K_{S}^{m}$ by bold alphabets as $\mathbf{x}=(\mathbf{x}^{(v)})_v \in K_{S}^{m},$ where $\mathbf{x}^{(v)}=(x_{1}^{(v)},\dots,x_{m}^{(v)}) \in K_{v}^{m}.$
     
    \item We denote by $\|\cdot\|_{v,2}$ the square of the usual Hermitian norm on $K_v^m$ if $v$ is complex, the usual Euclidean norm on $K_v^m$ if $v$ is real, and the sup norm defined by
    $$\|(x_1^{(v)}, \dots,x_m^{(v)})\|_{v,2} = \max_{i}|x_i^{(v)}|_v$$
    if $v$ is nonarchimedean. 
    
    \noindent Note that if $v$ is nonarchimedean and $\vec{x}=(x_1,\dots,x_m) \in K_v^m,$ then $\|\vec{x}\|_{v,2}=|\vec{x}|_v.$ 
    
    \item We will endow both $K_S$ and $K_S^m$  with the product metric on them respectively. More  precisely for any $x=\left(x^{(v)}\right)_v \in K_S$ and $\mathbf{x}=\left( \mathbf{x}^{(v)}\right)_v \in K_S^m$ the norms  are given by:
    
    $$\|x\|:=\sup_{v\in S} |x^{(v)}|_v$$
    and
    $$\| \mathbf{x} \|:= \sup_{v\in S} \|\mathbf{x}^{(v)}\|_{v,2}$$
    respectively.
    
    \item For $\mathbf{x}=(\mathbf{x}^{(v)})_v \in K_{S}^{m}$ we define the content $c(\mathbf{x})$ of $\mathbf{x}$ by:
    $$c(\mathbf{x})=\prod_{v\in S} \|\mathbf{x}^{(v)}\|_{v,2}.$$
    \item For the sake of convenience, we will keep the  notations $\|\cdot\|_{v,2}$ and $c(\cdot)$ to denote the norms and the content on the exterior powers $\bigwedge^r (K_v^m)$ and $\bigwedge^r (K_S^m),$ respectively. 
    
    \item Given a field $F$, the set of all $m \times n$ matrices over $F$ is denoted by $M_{m\times n}(F).$
    \item We define $M_{m\times n}(K_S)$ as:
    $$ M_{m\times n}(K_S):=\prod_{v \in S} M_{m\times n}(K_v).$$
    Elements of $M_{m\times n}(K_S)$ will be denoted by $g=\left(g^{(v)}\right)_v,$ where $g^{(v)} \in M_{m\times n}(K_v).$
    Similarly,
    $$\GL(m,K_S):=\prod_{v \in S}\GL(m,K_v)$$
    and
    $$\GL^1(m,K_S):=\left\{\left(g^{(v)}\right)_{v \in S} \in \GL(m,K_S) : \prod_{v\in S} \left|\text{det}\left(g^{(v)}\right)\right|_v =1\right\}.$$
    
    \item 
     \begin{equation}
         \begin{array}{rcl}
        \mathfrak{a}^{+}:= \{\mathbf{t}:=(t_{v_1}^{(1)},t_{v_1}^{(2)},\dots,t_{v_1}^{(m+n)},\dots,t_{v_{\ell}}^{(1)},t_{v_{\ell}}^{(2)},\dots,t_{v_{\ell}}^{(m+n)} ) \in \mathbb{R}_{+}^{m+n} \times \dots \times \mathbb{R}_{+}^{m+n}  :\\ t_{v_r}^{(i)} < 1 , \,\, t_{v_r}^{(j)} \geq 1\,\, \text{for} \,\, i=1,\dots,m;\,\, j=m+1,\dots,m+n;\,\, r=1,\dots,\ell  \\  \text{and} \,\, \prod_{v\in S} t_{v}^{(1)} \dots t_{v}^{(m+n)}=(\text{const}_K)^{m+n} \}, &  \\
         & 
    \end{array}
     \end{equation}
     where $\mathbb{R}^+$ is the set of positive real numbers.
     
     \item Now let $\mathcal{I}$ be a subset of $\mathfrak{a}^+$ defined by
     \begin{equation}
         \begin{array}{rcl}
        \mathcal{I}:= \{\mathbf{t} \in \mathfrak{a}^+  : t_{v_r}^{(1)}=\dots=t_{v_r}^{(m)}, t_{v_r}^{(m+1)}=\dots=t_{v_r}^{(m+n)} \,\, \text{for} \,\, r=1,\dots,\ell \}.
    \end{array}
     \end{equation}
    \item Let $S$ contain precisely all archimedian places of $K$. Given  $\mathbf{t} \in \mathcal{I},$ let $t_{v_r}:=t_{v_r}^{(1)}=\dots=t_{v_r}^{(m)}$ and $t'_{v_r}:=t_{v_r}^{(m+1)}=\dots=t_{v_r}^{(m+n)}$ for $r=1,\dots,\ell.$ Keeping this in mind, let  $\mathcal{I}_0$  be the subset of $\mathcal{I}$ defined as follows:
     \begin{equation}
         \begin{array}{rcl}
        \mathcal{I}_0 := \{ \mathbf{t} \in \mathcal{I}:t_{v_r}=t_{v_s} \,\, \text{and} \,\, t'_{v_r}=t'_{v_s} \,\, \text{for}\,\, r,s=1,\dots,\ell \}.
    \end{array}
     \end{equation}
     $\mathcal{I}_0$ will be called the \emph{central ray} of $\mathfrak{a}^+.$
    
    \item For $ \mathbf{t} \in \mathfrak{a}^{+}, $ define $\|\mathbf{t}\|_{\infty}= \displaystyle \max_{\tiny \begin{array}{rcl}
    i=1,\dots,m+n \\ r=1,\dots,\ell
    \end{array} } t_{v_r}^{(i)}.$
    
    \item For $i=1,\dots,\ell$, let $p_i : \mathfrak{a}^{+} \rightarrow \mathbb{R}_{+}^{m+n}$ be the projection of $\mathfrak{a}^{+}$ on the $i$-th factor $\mathbb{R}_{+}^{m+n}$ defined by
    $$p_i(\mathbf{t})=(t_{v_i}^{(1)},t_{v_i}^{(2)},\dots,t_{v_i}^{(m+n)}) \,\, \,\, \,\, \forall \,\, \mathbf{t} \in \mathfrak{a}^{+} .$$
    
    \item We will say that a subset $\mathcal{T}$ of $\mathfrak{a}^{+}$ is $S$-unbounded  if $p_i|_{\mathcal T}$ is unbounded for each $i=1,\dots,\ell.$
\end{itemize}
Now we define Dirichlet improvability in this context.  Let $\mathfrak{a}^{+}$ be as above and $\mathcal T $ be a $S$-unbounded subset of $\mathfrak{a}^{+}$.
\begin{definition}\label{defn:DI imp} Given $A=(A_v)_{v}\in M_{m\times n}(K_S),$ where $A_v$ is a $m \times n$ matrix over $K_v$ for each $v\in S$,  and $0 < \varepsilon < 1$. We say that Dirichlet's theorem can be $\varepsilon $-improved for $A$ along $ \mathcal{T} $ or use the notation $A \in DI_{\varepsilon}(m,n,K_S,\mathcal T)$ if there is $t_0>0$ such that given any $\varepsilon_{v}^{(1)}, \varepsilon_{v}^{(2)}, \dots , \varepsilon_{v}^{(m)},\delta_{v}^{(1)}, \delta_{v}^{(2)}, \dots, \delta_{v}^{(n)} \in K_v \setminus \{0\}$ for each $v \in S$  with 
\begin{equation}\label{equ:dimp}
    \begin{array}{rcl}
      \mathbf{t}:= ( \,\, |\varepsilon_{v_1}^{(1)}|_{v_1},\dots,|\varepsilon_{v_1}^{(m)}|_{v_1},|\delta_{v_1}^{(1)}|_{v_1} ,\dots,|\delta_{v_1}^{(n)}|_{v_1}    ,\dots,|\varepsilon_{v_{\ell}}^{(1)}|_{v_{\ell}},\dots,|\varepsilon_{v_{\ell}}^{(m)}|_{v_{\ell}}, \\ |\delta_{v_{\ell}}^{(1)}|_{v_{\ell}} ,\dots,|\delta_{v_{\ell}}^{(n)}|_{v_{\ell}} )  \in \mathcal{T}
    \end{array}
\end{equation}
and  $\|p_i|_{\mathcal{T}}(\mathbf{t})\|_{\infty} > t_0$ for each $i=1,\dots,\ell,$  there exists $\vec{x}=(x_1,\dots,x_n) \in \mathcal O_S^n \setminus \{\vec{0}\}$ and $\vec{y}=(y_1,\dots,y_m) \in \mathcal O_S^m$ satisfying
\begin{equation}\label{eqn:1.7}\left \{ \begin{array}{rcl} \|A_{v}^{(i)} \vec{x} - y_i \|_v  \leq \varepsilon  \|\varepsilon_{v}^{(i)}\|_v &\mbox{for} &i=1,2,...,m \\ \|x_j\|_v \leq  \varepsilon \|\delta_{v}^{(j)}\|_v &\mbox{for} &j=1,2,...,n\,, \end{array} \right.  
\end{equation}
for all $v \in S,$ where $A_{v}^{(i)}$ is the $i$th row of $A_v$ for $i=1,2,\dots, m$. In particular, a vector $ \mathbf{x}=(\mathbf{x}^{(v)})_v \in K_{S}^{n},$ where $\mathbf{x}^{(v)}=(x_1^{(v)},\dots,x_n^{(v)}) \in K_v^n$ is said to be \textit{Dirichlet $\varepsilon$-improvable along $\mathcal{T}$}
if the corresponding  matrix $A=([x_1^{(v)} \dots x_n^{(v)}])_{v \in S}\in DI_{\varepsilon}(1,n,K_S,\mathcal{T})$. When $\mathcal{T}=\mathcal{I} $, then  simply we will say $A$ is Dirichlet $\varepsilon$-improvable.  
 \end{definition}
 Singular vectors can be defined in the $S$-adic context in the exactly similar way to that of euclidean: 
\begin{definition}\label{defn:sing1} Let $m,n\in \mathbb{N}$ and $\mathcal{T}$ be a $S$-unbounded subset of $\mathfrak{a}^+$. We say $A=(A_v)_{v}\in M_{m\times n}(K_S)$, where $A_v\in M_{m\times n}(K_v)$, for each $v\in S$, is \emph{singular} along $\mathcal{T}$ if for all $\varepsilon \in (0,1)$, $A$ is Dirichlet $\varepsilon$-improvable along $\mathcal{T}$. In particular, a vector $ \mathbf{x}=(\mathbf{x}^{(v)}) \in K_S^n,$ where $\mathbf{x}^{(v)}=(x_1^{(v)},\dots,x_n^{(v)}) \in K_v^n$, for each $v\in S$ is said to be singular along $\mathcal{T}$ if $A=([x_1^{(v)} \cdots x_n^{(v)}])_{v \in S}$ is singular along $\mathcal{T}$. When $S$ consists of precisely all archimedian places of $K$ and $\mathcal{T}=\mathcal{I}_0 $, then  simply we will say $A$ is \emph{singular}. We denote the set of all  vectors of $K_S^n$  that are singular along $\mathcal{T}$ by $Sing(K_S,n,\mathcal{T})$.
\end{definition}

When $K$ is totally real and $S=$ the collection of all archimedian places of $K$, the notion of singular vectors has been introduced in \cite{DR}. 
 \begin{definition}{\cite[Definition 2.2]{DR}}\label{defn:sing2} Let $K$ be a totally real number field and $S$ is the collection of all archimedian places of $K$. We say $\mathbf{x}=(x_1,\dots,x_n) \in K_S^n$ is \emph{singular} if for every $c>0,$ for all sufficiently large $Q>0$ there exists $\mathbf{0} \neq \mathbf{q} \in \mathcal O_K^n,$ $q_0 \in \mathcal O_K$ satisfying the following system 

\begin{equation}\label{equ:6.1}\left \{ \begin{array}{rcl} \| \mathbf{q} \cdot \mathbf{x} + q_0 \|  <  \frac{c}{Q^n},  \\ \|\mathbf{q}\| \leq Q  \,. \end{array} \right.  
\end{equation}
  The set of all singular vectors in $K_S^n $, in this sense, is denoted by $Sing_S^n$. 
\end{definition}
Indeed, 
when $S$ is the set of all archimedian places of $K$, Definition \ref{defn:sing2} is equivalent to our definition. The following proposition establishes that.

\begin{lemma}\label{lemma:equiv}
 $Sing(K_S,n,\mathcal{I}_0)=Sing_S^n$, if $S$ contains precisely all normalized archimedean places of $K$. 
\end{lemma}
\begin{proof}
 By a standard algebraic argument using fractional ideals etc., one can easily see that in this case $\mathcal O_S=\mathcal O_K$. Let $\mathbf{x} \in Sing_S^n$. We will show that $\mathbf{x} \in Sing(K_S,n,\mathcal{I}_0).$ For that let us take $0< \varepsilon < 1.$  Suppose that  
 $\varepsilon_v,$ $\delta_v$ $\in K_v \setminus \{0\}$,  for each $v \in S$ with $|\varepsilon_v|_v=|\varepsilon_{v'}|_{v'}=\varepsilon_S,$ $|\delta_v|_v=|\delta_{v'}|_{v'}=\delta_S$ for all $v,v'\in S$ and $ \varepsilon_S <1$, $\delta_S \gg 1$ with
  $$ \varepsilon_S^{\ell} \delta_S^{n\ell}=(\text{const}_K)^{n+1}.$$

\noindent We choose  $c>0$ and $Q$ sufficiently large so that  

\begin{equation}\label{equ:6.2}
    \left \{ \begin{array}{rcl} \frac{c}{Q^n} < \varepsilon \varepsilon_S  \\ Q \leq \varepsilon \delta_S  \,, \end{array} \right.
\end{equation}   
holds. Since $\mathbf{x}$ is singular, for this chosen $c $ and sufficiently large $Q$ there exists  $\mathbf{0} \neq \mathbf{q} \in \mathcal O_K^n,$ $q_0 \in \mathcal O_K $ such that (\ref{equ:6.1}) holds. Observe that if $\mathbf{x}=(x_1,\dots,x_n) \in K_S^n$ and $x_i=(x_i^{(v)}) \in K_S$ for $i=1,\dots,n,$ then  (\ref{equ:6.1}) can be rewritten as
 \begin{equation}\label{equ:6.3}
     | q_1 x_1^{(v)}+q_2 x_2^{(v)}+ \dots +q_n x_n^{(v)} + q_0|_v \, \leq \, \frac{c}{Q^n} \text{ \,\,and } \displaystyle 
\max_{1\leq j \leq n} |q_j|_v \leq Q \,,
 \end{equation}
for all $v \in S$.\\

Now combining (\ref{equ:6.2}) and (\ref{equ:6.3}), we get that  $\mathbf{x} \in K_S^n$ is Dirichlet $\varepsilon$-improvable along $\mathcal{I}_0$. Therefore we have
$$Sing_S^n \subseteq Sing(K_S,n,\mathcal{I}_0).$$
The other containment  $Sing(K_S,n,\mathcal{I}_0) \subseteq Sing_S^n$ is obvious.
\end{proof}
\begin{remark}
    In the one dimensional case, singular points are precisely the elements of $K$ (\cite[Theorem 2.6]{DR}). In fact, when $K=\mathbb{Q}$, one can adopt the strategy of the proof of \cite[Theorem 2.4]{GG1} and easily see that, if $0<\varepsilon<\frac{1}{2}$, the real numbers for which the Dirichlet's theorem can be improved by $\varepsilon$ amount are precisely the rationals. The authors are grateful to Sanju Velani for drawing the attention of the second named author to this fact in an informal discussion. 
\end{remark}

\noindent Unless specified otherwise, we let  $\mathbf{t}$ stand for (\ref{equ:dimp}) now onward. Before going to our main result, let us recall some terminalogy from \cite{KLW}, \cite{KT} and \cite{KT-pre}. 

\begin{definition}\label{defn:besicovitch}Let $X$ be a metric space. We say that $X$ is $\textit{Besicovitch}$ if for any bounded subset $A$ of $X$ and for any family $\mathcal{B}$ of nonempty open balls in $X$ such that every $x \in A$ is a center of some open ball of $\mathcal{B}$, there exists a constant $N_X$ (called  Besicovitch constant of $X$) and there is a finite or countable subfamily $\{B_i\}$ of $\mathcal{B}$ such that 
\[1_A\leq \displaystyle \sum_i 1_{B_i} \leq N_X\,,\]
i.e. $A \subset \bigcup_i B_i$, and the multiplicity of that subcovering is at most $N_X$. 
\end{definition}
Throughout this article, $N_X$ will always represent Besicovitch constant of the metric space $X$.

\begin{example}
	\begin{enumerate}
		\item By Besicovitch's Covering Theorem \cite[Theorem 2.7]{Mat}, the Euclidean spaces $\mathbb R^n$ are Besicovitch spaces.
		\item   
		The space of our interest $X=\prod_{v \in S}K_v^{l_v},$ where $l_v \in \mathbb N,$ is also Besicovitch \cite[Section 1.6]{KT-pre}.
	\end{enumerate}

\end{example}

Let $X$ be a Besicovitch metric space and $(F,|\cdot|)$ is a valued field. Now for a subset $B$ of $X$ and a given function $f : X \rightarrow F$, we set 
$$\|f\|_B := \displaystyle \sup_{x\in B} |f(x)|.$$
Now if $f: X \rightarrow F$ is a measurable function, $\mu$ is a Radon measure on $X$ and  $B$ is a subset of $X$ with $\mu(B)>0$, we define $$\|f\|_{\mu,B} :=\|f\|_{B\cap \text{ supp }(\mu)} .$$ 

\begin{definition}\label{defn:C,alpha}
 Given $C,\alpha >0$, open $U \subset X $ and a Radon measure $\mu$ on $X$, $f: X \rightarrow F$ is said to be $(C,\alpha)-\textit{good}$ on $U$ with respect to $\mu$ if for every ball $B\subset U$
 with center in $\text{supp }(\mu)$ and any $\varepsilon >0$ one has
 \[\mu(\{x\in B: |f(x)| < \varepsilon\})\leq C\left(\frac{\varepsilon}{||f||_{\mu,B}}\right)^{\alpha} \mu(B)\,.\]
\end{definition}
The properties listed below follow immediately from  Definition \ref{defn:C,alpha}.
\begin{lemma}\label{lem:C,alpha} Let $X,U,\mu,F, f, C,\alpha,$ be as given above. Then we have 
\begin{enumerate}[(a)]
   \item $f$ \text{ is } $(C,\alpha)-good$ \text{ on  }$U$ with respect to $\mu \Longleftrightarrow \text{ so is } |f|$.
   \item $f$ is $(C,\alpha)-good$ on $U$  with respect to $\mu$ $\Longrightarrow$ so is $c f$ for all $c \in F$.
   \item \label{item:sup} $\forall i\in I, f_i$ are $(C,\alpha)-good$ on $U$ with respect to $\mu$ $\Longrightarrow$ so is $\sup_{i\in I} |f_i|$.
   \item $f$ is $(C,\alpha)-good$ on $U$ with respect to $\mu$, and $c_1\leq |\frac{f(x)}{g(x)}|\leq c_2$ for all $x \in U$  $\Longrightarrow g$ is $(C(\frac{c_2}{c_1})^{\alpha},\alpha)-$ good on $U$ with respect to $\mu$.
 \item Let $C_2 >1$ and 
 $\alpha _2>0$. $f$ is $(C_1,\alpha_1)-good$ on $U$ with respect to $\mu$ and $C_1 \leq C_2, \alpha_2 \leq \alpha_1 \Longrightarrow f$ is
 $(C_2,\alpha_2)-good$ on $U$ with respect to $\mu$.
  \end{enumerate}
\end{lemma}
We recall an important lemma from \cite[Section 2.3]{KT-pre}, which will be useful later in our discussion.
\begin{lemma}\label{lem:pr} For each $i=1,\dots,m,$ let $X_i$  be a  metric space, $\mu_i$ be a measure on $X_i,$ $U_i \subseteq X_i $ open, $f_i : X_i \rightarrow \mathbb R$ be a function which is $(C_i,\alpha_i)$-good on $U_i$ with respect to $\mu_i.$ Then the function $f : \prod_{i=1}^{m}X_i \rightarrow  \mathbb R$ defined by $f(x_1,\dots,x_m)=f(x_1)f(x_2) \dots f(x_m)$ is $(C,\alpha)$-good on $U_1 \times \dots \times U_m$ with respect to $\mu_1 \times \dots \times \mu_m,$ with $C$ and $\alpha$ explicitly computed in terms of $C_i, \alpha_i$ for $i=1,\dots,m.$
 
\end{lemma}
Now given a Randon measure $\mu$ on $X$, an open subset $U$ of $X$ with $\mu(U)>0$ and a map $\mathbf{f}=(f_1,f_2,\dots,f_m) : U \rightarrow F^m$, say that the pair $(\mathbf{f},\mu)$ is $(C,\alpha)-good$ on $U$ if any $F$-linear combination of $ 1,f_1,\dots,f_m $ is $(C,\alpha)$ good on $U$ with respect to $\mu.$

\begin{definition}
 Let $\mathbf{f}=(f_1,f_2, \dots,f_m)$ be a map from $U$ to $F^m,$ where $m \in \mathbb N.$ We say that $(\mathbf{f},\mu)$ is \emph{nonplanar} on $U$ if for any ball $B \subset U$ centered in $\text{supp}(\mu),$ the restrictions of $1,f_1,f_2, \dots,f_m$ to $B\cap \text{ supp }(\mu)$ are linearly independent over $F$; in other words, if  $\mathbf{f}(B\cap \text{ supp }(\mu))$ is not contained in any proper affine subspace of $F^m.$
\end{definition}

If $B=B(x,r),$ where $x\in X$ and $r>0$, is a ball in $X$ and $c>0$, we will use the notation $cB$ to denote the ball $B(x,cr).$
\begin{definition}
 Let $\mu$ be a Radon measure on $X$ and $D>0.$ $\mu$ is said to be $D$-$\textit{Federer}$ on an open subset $U$ of $X$ if for all balls $B$ centered at $\text{ supp }(\mu)$ with $3B\subset U$, one has 
 \[\frac{\mu(3B)}{\mu(B)}\leq D\,.\]
\end{definition}

\noindent We finish this section by stating the  main result of this paper, which says that for sufficiently small $\varepsilon>0$, $\mu$-almost every vector lying in the image of a good and nonplanar map is not Dirichlet $\varepsilon$-improvable along $\mathcal{T}$ for a large class of measures $\mu$ in this $S$-adic set-up.

\begin{theorem}\label{thm:1.4}Let $X= \prod_{v \in S}X_v,$ where $X_v=K_{v}^{l_v}$ with $l_v \in \mathbb N$ for $v\in S$, which is a Besicovitch metric space, $\mu_v$ be a measure on $X_v$ $(v \in S)$ so that   $\mu=\prod_{ v \in S} \mu_v$ is $D$-Federer measure on $X,$ and $U=\prod_{v \in S} U_v \subseteq X$ be open with $\mu(U)>0$ and let $ \mathbf{f}=\left(\mathbf{f}^{(v)}\right)_v:U \rightarrow K_S^n$  be a  map, where $\mathbf{f}^{(v)}$ are continuous maps from $U_v \subseteq K_{v}^{l_v}$ to $K_v^n$ such that $(\mathbf{f}^{(v)},\mu_v)$ is nonplanar on $U_v$ for all $v \in S$ and for each $v\in S,$ $(\mathbf{f}^{(v)},\mu_v)$ is $(C_v,\alpha_v)$-good on $U_v.$ Then there exists $\varepsilon_0=\varepsilon_0(n,C_{v_1}, \dots, C_{v_\ell}, \alpha_{v_1}, \dots, \alpha_{v_\ell}, D,N_X,K,S)>0$ such that for any $\varepsilon<\varepsilon_0$ 

\begin{equation}
  \mathbf{f}_{*}\mu(DI_{\varepsilon}(1,n,K_S,\mathcal{T}))=0 \,\,\text{ for any S-unbounded \,}\mathcal{T}\subseteq \mathfrak{a}^{+}\,.
 \end{equation}
\end{theorem}
  \begin{corollary}\label{cor:sing}
 Let the  hypothesis be same as that of Theorem \ref{thm:1.4}. Then we have,
 $$\mathbf{f}_* \mu \left(Sing(K_S,n,\mathcal{T})\right)=0, \text{ for any $S$-unbounded }\mathcal{T}\subseteq \mathfrak{a}^{+}.$$
 In particular, $\mathbf{f}_* \mu \left(Sing(K_S,n,\mathcal{I}_0)\right)=0.$
 \end{corollary}

In view of Lemma \ref{lemma:equiv}, it is now clear the above-mentioned Theorem \ref{thm:1.4} generalizes the following:
\begin{theorem}{\cite[Theorem 2.3]{DR}}\label{thm:sin} Suppose $X=\prod_{v \in S} X_v$ is a Besicovitch space and $\mu=\prod_{v\in S}\mu_v$ be a Federer measure on $X$ and let $\mathbf{f} : X \rightarrow K_S^n,$ $\mathbf{f}(x)=(\mathbf{f}_v(x_v))_{v\in S}$ be a continuous map such that $(\mathbf{f},\mu)$ is nonplanar for $\mu$-almost every point of $X$ and for each $v \in S,$ $(\mathbf{f}_v,\mu_v)$ is good for $\mu_v$-almost every point of $X_v.$ Then $\mathbf{f}_* \mu (Sing_S^n)=0$.  
\end{theorem}

The key idea to prove Theorem \ref{thm:1.4} is the so-called `quantitative nondivergence' technique, a generalization of non-divergence of unipotent flows on homogeneous spaces. To apply this technique first we need to translate our problem related to Dirichlet $\varepsilon$-improvability into certain properties of flows on some homogeneous spaces, which is the main theme of our next section.

\section{The correspondence}\label{section:2}
\noindent Define 
$$\Omega_{S,m}:= \left\{g\mathcal O_S^m : g \in \GL(m,K_S)\right\}$$
and $$\Omega_{S,m}^1:= \left\{\Lambda \in \Omega_{S,m} : \text{cov}(\Lambda)=(\sqrt{D_K})^m \right\}.$$

Let $G:= \GL^1\left(m,K_S\right)$ and $\Gamma:=\GL\left(m,\mathcal O_S\right),$ and let $\pi : G \rightarrow G/\Gamma$ be the quotient map. Then $G$ acts on $G/\Gamma$ by left translation via the rule $g\pi(h)=\pi(gh)$ for $g,h \in G.$ We define a map $\tau : M_{m \times n}(K_S) \longrightarrow \GL^1\left(m+n,K_S\right)$ given by
$$\tau \left(\left(A_v\right)_{v }\right)=\left(g_{A_v}\right)_{v },$$
where $$g_{A_v}:= \begin{bmatrix}
                 I_m& A_v\\0& I_n
\end{bmatrix}.$$
We put $\overline{\tau}:=\pi \circ \tau.$ Now the homogeneous space $\GL^1\left(m,K_S\right)/\GL\left(m,O_S\right)$ is naturally identified with $\Omega_{S,m}^1,$ so that we have $\overline{\tau}\left(\left(A_v\right)_{v}\right)=\left(g_{A_v}\right)_{v } \mathcal O_S^{m+n}.$
Define $$\Omega_{\varepsilon} :=\left\{ \Lambda \in \Omega_{S,m+n}:\delta(\Lambda) \geq \varepsilon\right\},$$
where $\delta : \Omega_{S,m+n} \longrightarrow \mathbb R_{+} $ defined by 
$$\delta(\Lambda) :=\min \left\{c(\mathbf{x}): \mathbf{x} \in \Lambda \setminus \{0\}\right\},$$
for all $\Lambda \in \Omega_{S,m+n}.$ 

\noindent Again we let $\varepsilon_{v}^{(1)}, \varepsilon_{v}^{(2)}, \dots , \varepsilon_{v}^{(m)},\delta_{v}^{(1)}, \delta_{v}^{(2)}, \dots, \delta_{v}^{(n)} \in K_v \setminus \{0\}$ for each $v \in S$ and define $g_{\varepsilon_v,\delta_v}$ by
$$g_{\varepsilon_v,\delta_v}:=\text{diag}\left(\left(\varepsilon_{v}^{(1)}\right)^{-1},\dots,\left(\varepsilon_{v}^{(m)}\right)^{-1},\left(\delta_{v}^{(1)}\right)^{-1},\dots,\left(\delta_{v}^{(n)}\right)^{-1}\right).$$
Then clearly $\left(g_{\varepsilon_v,\delta_v}\right)_{v } \in \GL\left(m+n,K_S\right).$ 

Now the following proposition, which may be regarded as an analogue of  \emph{`Dani correspondence'} over number fields, helps us to convert our problem regarding Dirichlet improvability to a problem of homogeneous dynamics.

\begin{proposition}
 $$A=\left(A_v\right)_{v} \in DI_{\varepsilon}\left(m,n,K_S,\mathcal T \right) \implies \left(g_{\varepsilon_v,\delta_v}\right)_{v} \left(g_{A_v}\right)_{v} \mathcal O_S^{m+n} \notin \Omega_{\left(m+n\right)^{|S_r|/2 + |S_c|} \displaystyle \varepsilon},$$
 for all $\varepsilon_{v}^{(1)}, \varepsilon_{v}^{(2)}, \dots , \varepsilon_{v}^{(m)},\delta_{v}^{(1)}, \delta_{v}^{(2)}, \dots, \delta_{v}^{(n)} \in K_v \setminus \{0\}$ (for each $v \in S$) with  $\mathbf{t}\in \mathcal T$ and $\|p_i|_{\mathcal{T}}(\mathbf{t})\|_{\infty} >>1 $ for each $i=1,\dots,\ell$, where $\mathbf{t}$ is given by (\ref{equ:dimp}).

\end{proposition}
\begin{proof}
Let $A=\left(A_v\right)_{v} \in DI_{\varepsilon}\left(m,n,K_S,\mathcal T\right).$ Then  given any $\varepsilon_{v}^{(1)}, \varepsilon_{v}^{(2)}, \dots$ , $\varepsilon_{v}^{(m)},\delta_{v}^{(1)}, \delta_{v}^{(2)}, \dots, \delta_{v}^{(n)}$ in $ K_v \setminus \{0\}$, for $v\in S$, with  $\mathbf{t} \in \mathcal T$, where $\mathbf{t}$ is given by (\ref{equ:dimp}),  and $\|p_i|_{\mathcal{T}}(\mathbf{t})\|_{\infty} >>1 $, for each $i=1,\dots,\ell$, there exists $\vec{x}=(x_1,\dots,x_n) \in \mathcal O_S^n \setminus \{0\}$ and $\vec{y}=(y_1,\dots,y_m) \in \mathcal O_S^m$ such that   (\ref{eqn:1.7}) holds. Now observe that

\begin{equation}\label{eqn:lag}\begin{array}{rcl}\displaystyle \left(g_{\varepsilon_v,\delta_v}\right)_{v} \left(g_{A_v}\right)_{v} \mathcal O_S^{m+n}= \{ ( {\varepsilon_{v}^{(1)}}^{-1}(A_v^{(1)} \vec{b} - a_1), \dots ,{\varepsilon_{v}^{(m)}}^{-1}(A_v^{(m)} \vec{b} - a_m), \\ \displaystyle {\delta_{v}^{(1)}}^{-1} b_1,\dots, {\delta_{v}^{(n)}}^{-1} b_n)_{v }  \\ \displaystyle : \mathbf{z}=(\mathbf{z}^{(v)})_{v }\in \mathcal O_S^{m+n} \} \,,\end{array}
 \end{equation}
where $\mathbf{z}^{(v)}=(\vec{a},\vec{b}) \in K^{m+n}$ for each $v\in S,$   $\vec{a}=(a_1,\dots,a_m) \in \mathcal{O}_S^m$ and $\vec{b}=(b_1,\dots,b_n) \in \mathcal{O}_S^n.$

Now consider $\vec{x}\in \mathcal O_S^n \setminus \{0\}$ and $\vec{y} \in \mathcal O_S^m$ satisfying (\ref{eqn:1.7}) as elements of $ K_S^n \setminus \{0\}$ and $K_S^m$ respectively, and denote them as $\mathbf x=(\mathbf{x}^{(v)})_{v} \in \mathcal O_S^n \setminus \{0\}$ and $\mathbf{y}=(\mathbf{y}^{(v)})_{v } \in \mathcal O_S^m  $  respectively. Then $(\mathbf{x},\mathbf{y}) \in \mathcal O_S^{m+n}$, and for this point let $\mathbf{w}=(\mathbf{w}^{(v)})_{v}$ be the corresponding point on the lattice $\left(g_{\varepsilon_v,\delta_v}\right)_{v} \left(g_{A_v}\right)_{v} \mathcal O_S^{m+n}$. Now using (\ref{equ:dimp}) and (\ref{eqn:1.7}) we get that
\begin{align*}
    \displaystyle c(\mathbf{w})= \prod_{v \in S } \|(\mathbf{w}^{(v)})\|_{v,2 } &\displaystyle = \left(\prod_{v \in S_r}\|\mathbf{w}^{(v)}\|_{v,2} \right) \left(\prod_{v \in S_c}\|\mathbf{w}^{(v)}\|_{v,2} \right) \left(\prod_{v \in S_f}\|\mathbf{w}^{(v)}\|_{v,2} \right) \\&
     \displaystyle < \left(\prod_{v \in S_r} \sqrt{(m+n){\varepsilon}^2} \right) \left(\prod_{v \in S_c} (m+n){\varepsilon}^2 \right) \left(\prod_{v \in S_f} \varepsilon  \right) \\& \displaystyle = \left(\left(m+n\right)^{|S_r|/2} \varepsilon^{|S_r|} \right) \left(\left(m+n\right)^{|S_c|} \varepsilon^{ 2|S_c|}\right) \left(\varepsilon^{|S_f|}\right) \\ & \displaystyle = \left(m+n\right)^{|S_r|/2 + |S_c|} \varepsilon^{|S_r|+2|S_c|+|S_f|} \\& \displaystyle < \left(m+n\right)^{|S_r|/2 + |S_c|} \varepsilon,
\end{align*}

\noindent since $d_v \geq 1,$ $|S_r|+2|S_c|+|S_f| \geq |S| \geq 1$ and $0 < \varepsilon <1,$ where for any finite set $A,$ $|A|$ denotes the number of elements of $A.$

Therefore $ \delta(\Lambda) < \left(m+n\right)^{|S_r|/2 + |S_c|} \varepsilon ,$ where $\Lambda=\left(g_{\varepsilon_v,\delta_v}\right)_{v} \left(g_{A_v}\right)_{v} \mathcal O_S^{m+n} $ and so we finally get our desired result.
 
\end{proof}

Now using the above proposition we can say that given a Besicovitch mectric space $X,$ a Federer measure $\mu$ on $X,$ a open subset $U$ of $X$ with $\mu(U)>0$ and a map $F: U \longrightarrow M_{m \times n}(K_S)$ to prove $F_{*} \mu\left(DI_{\varepsilon}(m,n,K_S,\mathcal T)\right)=\mu\left(F^{-1}\left(DI_{\varepsilon}(m,n,K_S,\mathcal T)\right)\right)=0$ it is enough to show that 
\begin{equation}\label{equ:1.8}
    \mu\left(F^{-1}\left(\left\{(A_v)_{v} \in M_{m \times n}(K_S) : \left(g_{\varepsilon_v,\delta_v}\right)_{v} \left(g_{A_v}\right)_{v} \mathcal O_S^{m+n} \notin \Omega_{\left(m+n\right)^{|S_r|/2 + |S_c|} \displaystyle \varepsilon}\right\}\right)\right)=0,
\end{equation}
for all $\varepsilon_{v}^{(1)}, \varepsilon_{v}^{(2)}, \dots , \varepsilon_{v}^{(m)},\delta_{v}^{(1)}, \delta_{v}^{(2)}, \dots, \delta_{v}^{(n)} \in K_v \setminus \{0\}$ (for each $v\in S$) with $\mathbf{t} \in \mathcal T$ and $\|p_i|_{\mathcal{T}}(\mathbf{t})\|_{\infty} >>1 $ for each $i=1,\dots,\ell,$ where $\mathbf{t}$ is given by (\ref{equ:dimp}).

\noindent Now we suppose that for any ball $B \subseteq U$ centred in $\mbox{supp}\,(\mu)$ there exists a $c\in (0,1)$ such that 
\begin{equation}\label{equ:1.9}\begin{array}{lcl}
    \displaystyle \mu \left(B \cap F^{-1}\left(\left\{(A_v)_{v} \in M_{m \times n}(K_S) : \left(g_{\varepsilon_v,\delta_v}\right)_{v} \overline{\tau}\left(\left(A_v\right)_{v}\right) \notin \Omega_{\left(m+n\right)^{|S_r|/2 + |S_c|} \displaystyle \varepsilon}\right\}\right)\right) \\ \displaystyle = \mu\left(\left\{\mathbf{x} \in B : \left(g_{\varepsilon_v,\delta_v}\right)_{v} \overline{\tau}\left(F(\mathbf{x})\right) \notin \Omega_{\left(m+n\right)^{|S_r|/2 + |S_c|} \displaystyle \varepsilon}\right\}\right) \leq c \mu(B),
    
\end{array}\end{equation}
holds for all $\varepsilon_{v}^{(i)}$'s and $\delta_{v}^{(j)}$'s satisfying the aforementioned conditions.

\noindent If (\ref{equ:1.9}) holds true, then for any ball $B \subseteq U$ centered in $\mbox{supp}\,(\mu)$ we have 
\begin{equation}\label{equ:1.10}\begin{array}{lcl}
\displaystyle \frac{\mu \left(B \cap F^{-1} \left(\left\{(A_v)_{v} \in M_{m \times n}(K_S) : \left(g_{\varepsilon_v,\delta_v}\right)_{v} \overline{\tau}\left(\left(A_v\right)_{v}\right) \notin \Omega_{ \left(m+n\right)^{|S_r|/2 + |S_c|} \displaystyle \varepsilon} \right\}\right)\right)}{\mu(B)} \\ \displaystyle \leq \frac{c\mu(B)}{\mu(B)}=c<1,

\end{array}\end{equation}
for all $\varepsilon_{v}^{(1)}, \varepsilon_{v}^{(2)}, \dots , \varepsilon_{v}^{(m)},\delta_{v}^{(1)}, \delta_{v}^{(2)}, \dots, \delta_{v}^{(n)} \in K_v \setminus \{0\}$ (for each $v\in S$) with $\mathbf{t} \in \mathcal T$ and $\|p_i|_{\mathcal{T}}(\mathbf{t})\|_{\infty} >>1 $ for each $i=1,\dots,\ell.$ 

\noindent Since (\ref{equ:1.10}) holds for any ball $B \subseteq U,$ it follows that no pont $\mathbf{x} \in U \cap \mbox{supp}(\mu)$ is a point of density of the sets 
$$ F^{-1}\left(\left\{(A_v)_{v} \in M_{m \times n}(K_S) : \left(g_{\varepsilon_v,\delta_v}\right)_{v} \overline{\tau}\left((A_v)_{v}\right) \notin \Omega_{ \left(m+n\right)^{|S_r|/2 + |S_c|} \displaystyle \varepsilon} \right\}\right),$$
for all $\varepsilon_{v}^{(1)}, \varepsilon_{v}^{(2)}, \dots , \varepsilon_{v}^{(m)},\delta_{v}^{(1)}, \delta_{v}^{(2)}, \dots, \delta_{v}^{(n)} \in K_v \setminus \{0\}$ (for each $v\in S$) with $\mathbf{t} \in \mathcal T$ and $\|p_i|_{\mathcal{T}}(\mathbf{t})\|_{\infty} >>1 $ for each $i=1,\dots,\ell,$ 
which gives us (\ref{equ:1.8}), because of the following version of the Lebesgue density theorem for our concerned metric space $X=\prod_{v\in S} K_v^{l_v},$ where $l_v \in \mathbb N,$ endowed with the product metric.

\begin{lemma}\label{lemma:den}If $X=\prod_{v\in S} K_v^{l_v},$ where $l_v \in \mathbb N,$ endowed with the product metric is our concerned metric space, $\mu$ be a Radon measure on $X$ and $\Omega $ is any measurable subset of $X.$ Then  almost 
 every point $\mathbf{x}\in \Omega\,\cap \text{ supp }(\mu)$ is a point of density of $\Omega$, that is, 
 $$ \displaystyle \lim_{\tiny \begin{array}{rcl}\mu (B)\rightarrow 0\\ \mathbf{x}\in B\end{array}}\frac{\mu(B\cap \Omega)}{\mu(B)}=1\,.$$
\end{lemma}

The proof of  Lemma \ref{lemma:den} goes along the same line as the Lebesgue Density Theorem for Euclidean spaces with some appropriate modifications. Hence we leave it to the reader.

As the space of our interest $X$ is locally compact, Hausdorff and second countable, to prove  Theorem \ref{thm:1.4}, it is enough to prove the theorem locally, that is, after choosing a suitable   $\varepsilon_0 >0$ it is enough to show that  for all $\mathbf{w} \in U \cap \text{ supp }(\mu),$  there exists a ball $B' \subseteq U$ containing $\mathbf{w}$ such that 
\begin{equation}
    \mu\left(B' \cap \mathbf{f}^{-1}\left(DI_{\varepsilon}(1,n,K_S,\mathcal{T})\right)\right)=\mu\left(\left\{\mathbf{x} \in B': \mathbf{f}(\mathbf{x}) \in DI_{\varepsilon}(1,n,K_S,\mathcal{T})\right\}\right)=0
\end{equation}
for all $\varepsilon < \varepsilon_0.$ And now from the above discussion it follows that Theorem \ref{thm:1.4} is an immediate consequence of the following proposition.

\begin{proposition}\label{main prop}
 For any $l_v,n \in \mathbb{N},$  there exists $\tilde C $ with the following property: whenever a ball $ B=\prod_{v \in S}B(x_v,r)$ centered in $\mbox{supp}\,(\mu)$, a  Radon measure  $\mu=\prod_{v\in S}\mu_v$ on $X=\prod_{v \in S}K_v^{l_v}$ which is D-Federer 
 on $\tilde B:= 3^{n+1} B=\prod_{v \in S}B(x_v,3^{n+1}r)$ with $\mu_v$ is a measure on $K_v^{l_v}$ 
 and $\mathbf{f} = \left(\mathbf{f}^{(v)}\right)_v:\tilde {B} \longrightarrow K_S^n, $   be a given  map, where $\mathbf{f}^{(v)}=\left(f_1^{(v)},\dots,f_n^{(v)}\right)$ are continuous maps from $\tilde{B}_v:=B(x_v,3^{n+1}r) \subseteq K_v^{l_v}$ to $K_v^n$ for each $v\in S$, so that
 \begin{enumerate}[(i)]
\item for some $C_v,\alpha_v>0,$ any linear combination of $1,f_1^{(v)},\dots,f_n^{(v)}$ is $(C_v,\alpha_v)$ good on $\tilde{B}_v$ with respect to $\mu_v$ for all $v \in S$ and\,,
\item the restrictions of $1,f_1^{(v)},\dots,f_n^{(v)}$ to $B(x_v,r) \cap  \text{ supp }(\mu_v)$ are linearly independent over $K_v$ for all $ v\in S $.
 
 \end{enumerate}
\noindent Then  for any $ 0 < \varepsilon \leq \frac{\rho}{\sqrt{D_K}}$ we have
\begin{equation}\label{eqn:prop}
\mu\left(\left\{\mathbf{x}\in B\,:\, \left(g_{\varepsilon_v,\delta_v}\right)_{v} \overline{\tau}\left(\mathbf{f}(\mathbf{x})\right) \notin \Omega_{\left(n+1\right)^{|S_r|/2 + |S_c|} \displaystyle \varepsilon}   \right\}\right)\leq \tilde C \varepsilon ^{\alpha} \mu(B)\,,
\end{equation}
for all $\varepsilon_{v}^{(1)}, \varepsilon_{v}^{(2)}, \dots , \varepsilon_{v}^{(m)},\delta_{v}^{(1)}, \delta_{v}^{(2)}, \dots, \delta_{v}^{(n)} \in K_v \setminus \{0\}$ (for each $v\in S$) with $\mathbf{t} \in \mathcal T$ and $\|p_i|_{\mathcal{T}}(\mathbf{t})\|_{\infty} >>1 $ for each $i=1,\dots,\ell,$  where $\mathbf{t}$ is given by (\ref{equ:dimp}), and $\rho$ is given by (\ref{equ:rho}). $\tilde{C} $ depends only on $(n, C_{v_1}, \dots, C_{v_\ell}, \alpha_{v_1}, \dots, \alpha_{v_\ell},D,N_X,K,S)$ and the expression for $\tilde{C}$ is given by (\ref{equ: ctil}),  
\end{proposition}

Now Theorem \ref{thm:1.4} follows immediately from Proposition \ref{main prop}. We choose our $ 0 < \varepsilon_0 \leq \frac{\rho}{\sqrt{D_K}}$ so that $\tilde{C} \varepsilon_{0}^{\alpha} < 1.$ In view of  the expression of $\tilde{C}$ (\ref{equ: ctil}), it is easy to see that $\varepsilon_0$ depends only on $(n, C_{v_1}, \dots, C_{v_\ell}, \alpha_{v_1}, \dots, \alpha_{v_\ell},D,N_X,K,S).$ Now consider $\mathbf{x} \in U \cap  \mbox{supp}(\mu)$ and choose a ball $B'$ containing $\mathbf{x}$ such that $B' \subseteq \tilde{B'} :=3^{n+1}B' \subseteq U.$ Next choose any ball $B \subseteq B'$ with centre in $\mbox{supp}(\mu)$ and consider $\tilde{B}:=3^{n+1}B.$ Since by the hypothesis of Theorem \ref{thm:1.4}, for each $v \in S,$ $\left(\mathbf{f}^{(v)},\mu_v\right)$ is  nonplanar and $(C_v,\alpha_v)$-good on $U_v$ with respect to $\mu_v,$ conditions (i) and (ii) of Proposition \ref{main prop} holds immediately. Now we set $c=\tilde{C}\varepsilon_{0}^{\alpha} <1,$ and observe that    \eqref{equ:1.9} holds whenever $0 < \varepsilon < \varepsilon_0,$ in view of Proposition \ref{main prop}. This completes the proof of Theorem \ref{thm:1.4}, assuming Proposition \ref{main prop}.

Now we need to prove Proposition \ref{main prop}. We will do it using the so called  `Quantitative nondivergence' technique and our next section is all about that.

\section{Quantitative nondivergence and the proof of Proposition \ref{main prop}} \label{section:5}

We first recall the $S$-arithmetic quantitative nondivergence theorem from \cite{KT-pre}.
\subsection{Quantitative nondivergence} 
We denote the set of all primitive submodules of $\mathcal O_S^m$ by $\frk P(\mathcal O_S,m).$ We recall the following theorem from  \S 6.3 of \cite{KT-pre}.

\begin{theorem}\label{thm:qn} Let $X$ be a Besicovitch metric space, $\mu$ be a $D$-Federer measure on $X.$ For $m \in \mathbb N,$ let a ball $B \subseteq X$ and a continuous map $h : \tilde B \longrightarrow \GL(m,K_S)$ be given, where $\tilde B := 3^m B.$ Now suppose that for some $C, \alpha >0$ and $0 < \rho <1 $ one has
\begin{enumerate}[(C1)]
    \item \label{item:C1} for every $\Delta \in \frk P(\mathcal O_S,m),$ the function $\text{cov}\left(h\left(\cdot\right)\Delta\right)$ is $(C,\alpha)$-good on $\tilde B$ with respect to $\mu;$
    \item \label{item:C2} for every $\Delta \in \frk P(\mathcal O_S,m),$ $\|\text{cov}\left(h\left(\cdot\right)\right)\Delta\|_{\mu,B} \geq \rho.$
\end{enumerate}
Then for any positive $\varepsilon \leq \rho / \sqrt{D_K}$ one has 
\begin{equation}\label{equ:qn}
\displaystyle  \mu\left(\left\{x\in B\,:\, \delta(h(x)\mathcal O_S^m)\, <\, \varepsilon \,\, \right\}\right)\leq mC\left(N_X D^2\right)^m \left(\frac{\varepsilon \sqrt{D_K}}{\rho}\right)^{\alpha} \mu(B)\,.
\end{equation}
\end{theorem}
For more details and the proof of Theorem \ref{thm:qn}, 
see (\cite{KT-pre}, \S 6, Theorem 6.3).

\subsection{The proof of Proposition \ref{main prop}}
From the definition of $\Omega_{\varepsilon}$, it follows  that 
\[ \left(g_{\varepsilon_v,\delta_v}\right)_{v} \overline{\tau}\left( \mathbf{f}(\mathbf{x})\right) \notin \Omega_{\left(n+1\right)^{|S_r|/2 + |S_c|} \displaystyle \varepsilon} \Longleftrightarrow 
 \delta \left(\left(g_{\varepsilon_v,\delta_v}\right)_{v} \left(g_{\mathbf{f}^{(v)}(x_v)}\right)_{v} \mathcal O_S^{n+1} \right) < \left(n+1\right)^{|S_r|/2 + |S_c|} \displaystyle \varepsilon  \,.
\]

\noindent We will use Theorem \ref{thm:qn} in the setting
\[\begin{array}{rcl}
  X=\prod_{v \in S} X_v=\prod_{v \in S}K_v^{l_v},m=n+1;\\ \mu, B, C_v, \alpha_v \,\,\mbox{and}\,\,D\,\,\mbox{as in Proposition \ref{main prop}};\\
  h(\mathbf{x})= \left(g_{\varepsilon_v,\delta_v}\right)_{v} \left(g_{\mathbf{f}^{(v)}(x_v)}\right)_{v} \,\,\forall \,\, \mathbf{x}=\left(x_v\right)_v\in \tilde{B}\,,\,\end{array}
 \] where   $$g_{\varepsilon_v,\delta_v}:=\text{diag}\left(\left(\varepsilon_{v}^{(1)}\right)^{-1},\left(\delta_{v}^{(1)}\right)^{-1}, \left(\delta_{v}^{(2)}\right)^{-1} ,\dots,\left(\delta_{v}^{(n)}\right)^{-1}\right) \in \GL(n+1,K_v),$$
 with  $\varepsilon_{v}^{(i)}$'s and $\delta_{v}^{(j)}$'s satisfying the aforementioned conditions and 
 
$$g_{\mathbf{f}^{(v)}(x_v)}:= \begin{bmatrix}
                 1 & \mathbf{f}^{(v)}(x_v) \\0 & I_n
\end{bmatrix} \in \GL(n+1,K_v).$$

Now to get our desired result equation  (\ref{eqn:prop}), we just need to verify the two conditions (C\ref{item:C1}) and (C\ref{item:C2})  of Theorem \ref{thm:qn}. Observe that to check (C\ref{item:C1}) and (C\ref{item:C2}), we first need  explicit expressions for functions $\mathbf{x} \mapsto \text{cov} \left(h(\mathbf{x}) \Delta \right)$ in our setting. 
 
    We recall that given any $\Delta \in \frk P(\mathcal O_S,n+1)$ of rank $j$ there exists an element $\mathbf{w} \in \bigwedge^j (\mathcal{O}_S^{n+1})$ such that 
 $$\text{cov}(\Delta)=a_{K} c(\mathbf{w}),$$
 where $a_{K} \in \mathbb R$ is a constant depending on $K.$ Therefore, we have $\text{cov}(h(\textbf{x})\Delta)=a_{K}c(h(\mathbf{x})\mathbf{w}).$
 
  Before proceeding further let us fix some notations. We take $\mathbf{e}_0,\mathbf{e}_1, \dots, \mathbf{e}_n \in K_S^{n+1}$ as the standard basis of $K_S^{n+1}$ over $K_S,$ where
  $$\mathbf{e}_i=\left( \mathbf{e}_{i}^{(v)}\right)_v, \,\,  \,\, \mathbf{e}_{i}^{(v)} \in K_v^{n+1} \,\, \mbox{for each} \,\, i=0,\dots,n ;$$
and $\left\{ \mathbf{e}_{i}^{(v)}\right\}_{i=0}^{n}$ forms the standard basis of $K_v^{n+1}$ over $K_v$ for each $v\in S.$ Now let $I=\{i_1,i_2,\dots,i_j\} \subset \{0,1, \dots,n \}$ with $i_1 < i_2 < \dots < i_j $ and $\mathbf{e}_I:= \mathbf{e}_{i_1} \wedge \mathbf{e}_{i_2} \dots \wedge \mathbf{e}_{i_j} \in \bigwedge^j(K_S^{n+1}).$ Then the standard basis of $\bigwedge^j(K_S^{n+1})$ is given by 

$$ \left \{ \mathbf{e}_I : I=\{i_1,i_2,\dots,i_j\} \subset \{0,1, \dots,n \} \,\, \mbox{with} \,\, i_1 < i_2 < \dots < i_j \right \},$$ 

\noindent and for any element $\mathbf{b}= \displaystyle\sum_{\tiny \begin{array}{rcl}I\subseteq \{0,1,...,n\},\\ \# I=j\end{array}} b_I \mathbf{e}_I  ,$ we define $\|\mathbf{b}\| = \displaystyle \max_{I} \|b_I\|.$

So we can write the earlier mentioned $\mathbf{w} \in \bigwedge^j (\mathcal{O}_S^{n+1})$ as $\mathbf{w}= \sum_{I} w_I \mathbf{e}_I,$ where $w_I \in \mathcal{O}_S.$ Now we will find the expession of $ \left(g_{\varepsilon_v,\delta_v}\right)_{v} \left(g_{\mathbf{f}^{(v)}(x_v)}\right)_{v} \mathbf{w}$ componentwise. Note that $g_{\mathbf{f}^{(v)}(x_v)}$ leaves $\mathbf{e}_{0}^{(v)}$
 invariant and sends $\mathbf{e}_{i}^{(v)}$ to $f_{i}^{(v)}(x_v) \mathbf{e}_{0}^{(v)} + \mathbf{e}_{i}^{(v)}$ for all $i=1,\dots,n.$ Therefore, we have 
\begin{equation} \label{eqn:5.1}  \left(g_{\mathbf{f}^{(v)}(x_v)}\right) \left(\mathbf{e}_{I}^{(v)} \right) = \left \{ \begin{array}{lcl}\displaystyle  \mathbf{e}_{I}^{(v)} & \mbox{if} \,\, 0\in I \\
                         \displaystyle  \mathbf{e}_{I}^{(v)} + \sum_{i\in I} \pm f_{i}^{(v)}(x_v) \mathbf{e}_{I \cup \{0\} \setminus \{i\}}^{(v)}             &\mbox{otherwise} .
                      
                        \end{array} \right.\end{equation} 
and 

\begin{equation} \label{eqn:5.2}  \left(g_{\varepsilon_v,\delta_v} \right) \left(g_{\mathbf{f}^{(v)}(x_v)}\right) \left(\mathbf{e}_{I}^{(v)} \right) = \left \{ \begin{array}{lcl}\displaystyle {\varepsilon_{v}^{(1)}}^{-1} \left(\prod_{i \in I \setminus \{0\}} {\delta_{v}^{(i)}}^{-1} \right) \mathbf{e}_{I}^{(v)} & \mbox{if} \,\, 0\in I \\
                         \displaystyle  \left(\prod_{i \in I} {\delta_{v}^{(i)}}^{-1} \right) \mathbf{e}_{I}^{(v)} \pm {\varepsilon_{v}^{(1)}}^{-1} \sum_{i\in I} \left(\prod_{r \in I \setminus \{i\}} {\delta_{v}^{(r)}}^{-1} \right) f_{i}^{(v)}(x_v) \mathbf{e}_{I \cup \{0\} \setminus \{i\}}^{(v)}             &\mbox{otherwise} .
                         \end{array} \right.\end{equation} 

\noindent Thus $\left(g_{\varepsilon_v,\delta_v}\right)_{v} \left(g_{\mathbf{f}^{(v)}(x_v)}\right)_{v} \mathbf{w}$ is given by 
 \begin{equation}\label{equ:5.3}
    \begin{array}{rcl}
      \left(\left(g_{\varepsilon_v,\delta_v}\right)_{v} \left(g_{\mathbf{f}^{(v)}(x_v)}\right)_{v} \mathbf{w}\right)^{(v)}= \displaystyle \sum_{0 \notin I}  \left(\prod_{i \in I} {\delta_{v}^{(i)}}^{-1} \right) w_{I} \mathbf{e}_{I}^{(v)} +\\ {\varepsilon_{v}^{(1)}}^{-1} \sum_{0 \in I } \left( \left(\prod_{i \in I \setminus \{0\}} {\delta_{v}^{(i)}}^{-1} \right) w_{I} +    \displaystyle \sum_{i\notin I} \pm \left(\prod_{r \in I \cup \{i\} \setminus \{0\}} {\delta_{v}^{(r)}}^{-1} \right) f_{i}^{(v)}(x_v) w_{I \cup \{i\} \setminus \{0\}}   \right) \mathbf{e}_{I}^{(v)}.
         
    \end{array}
\end{equation}
 Now we will check the conditions (C\ref{item:C1}) and (C\ref{item:C2}).   

$\bullet$ \textbf{Checking (C\ref{item:C1}):} Due to (\ref{equ:5.3}), we can have the explicit  expression for $\text{cov} (h(\mathbf{x}) \Delta)$ and that is given by 
$$\text{cov} \left( h(\mathbf{x}) \Delta \right)= a_{K} c \left( \left(g_{\varepsilon_v,\delta_v}\right)_{v} \left(g_{\mathbf{f}^{(v)}(x_v)}\right)_{v} \mathbf{w} \right) = a_K \prod_v \left \|\left(\left(g_{\varepsilon_v,\delta_v}\right)_{v} \left(g_{\mathbf{f}^{(v)}(x_v)}\right)_{v} \mathbf{w}\right)^{(v)} \right \|_{v,2}.$$
In view of (\ref{equ:5.3}), it is easy to observe that the components of all the coordinates of $$\left(\left(g_{\varepsilon_v,\delta_v}\right)_{v} \left(g_{\mathbf{f}^{(v)}(x_v)}\right)_{v} \mathbf{w}\right)^{(v)}$$ are linear combinations of $1,f_{1}^{(v)}, \dots, f_{n}^{(v)}.$ Then the condition (C\ref{item:C1}) immediately follows from hypothesis $(i)$ of Proposition \ref{main prop}, Lemma \ref{lem:C,alpha}(b,c)  and Lemma \ref{lem:pr}. So there exists $C, \alpha >0$ such that $ \mathbf{x} \mapsto \text{cov} (h(\mathbf{x}) \Delta) $  is $(C,\alpha)$-good on $\tilde{B}$ with respect to $\mu,$ where $C$ and $\alpha$ explicitly computed in terms of $ C_{v_1}, \dots, C_{v_{\ell}}, \alpha_{v_1}, \dots, \alpha_{v_{\ell}}.$

$\bullet$ \textbf{Checking (C\ref{item:C2}):} Now by the compactness of the unit sphere in $K_v^{n+1}$ and hypothesis $(ii)$ of Proposition (\ref{main prop}), there exists $\rho_v >0$ such that for any $\mathbf{a}=(a_0,a_1,\dots,a_n) \in K_v^{n+1}$ we have 
\begin{equation}\label{equ:5.4}
\displaystyle \sup_{\tiny x_v \in B_v\,\cap \,\mbox{supp}\,(\mu_v)} |a_0 + a_1 f_{1}^{(v)}(x_v) + \dots +a_n f_{n}^{(v)}(x_v) |_v \geq \rho_v \|\mathbf{a}\|_{v,2},
\end{equation}
where $B=\prod_{v\in S}B_v.$

\noindent Now using (\ref{equ:5.3}) and (\ref{equ:5.4}), we get that 
 \begin{equation}\label{equ:5.5}
     \displaystyle \sup_{\tiny x_v \in B_v\,\cap \,\mbox{supp}\,(\mu_v)} \left \|\left(\left(g_{\varepsilon_v,\delta_v}\right)_{v} \left(g_{\mathbf{f}^{(v)}(x_v)}\right)_{v} \mathbf{w}\right)^{(v)} \right \|_{v,2} \geq \rho_v |\varepsilon_v^{(1)}|_{v}^{-1} \left(\prod_{i=1}^{n} |\delta_{v}^{(i)}|_{v}^{-1} \right) \max_I |w_I|_{v} ,
 \end{equation}
since  $|\delta_{v}^{(i)}|_{v} \geq 1$ for all $i=1,\dots,n$ and each $v \in S.$ Therefore, we have

\begin{align}\label{equ:5.6}
    \displaystyle \sup_{\tiny x \in B\,\cap \,\mbox{supp}\,(\mu)} c \left(\left(g_{\varepsilon_v,\delta_v}\right)_{v} \left(g_{\mathbf{f}^{(v)}(x_v)}\right)_{v} \mathbf{w}\right) & \geq \displaystyle \prod_{v\in S} \left( \rho_v |\varepsilon_v^{(1)}|_{v}^{-1} \left(\prod_{i=1}^{n} |\delta_{v}^{(i)}|_{v}^{-1} \right) \max_I |w_I|_{v} \right) \nonumber \\ &\displaystyle \geq  \left( \prod_{v \in S} |\varepsilon_{v}^{(1)}|_{v} |\delta_{v}^{(1)}|_{v}   \dots |\delta_{v}^{(n)}|_{v} \right)^{-1} \max_I c(w_I) \left(\prod_{v\in S} \rho_v \right),
\end{align}

\noindent and hence
\begin{equation}\label{equ:cov}
  \|\mbox{cov}\left(h\left(\mathbf{x}\right)\right)\Delta\|_{\mu,B}    \displaystyle \geq  \left(\sqrt{D_K}\right)^j \left( \prod_{v \in S} |\varepsilon_{v}^{(1)}|_{v} |\delta_{v}^{(1)}|_{v} \dots |\delta_{v}^{(n)}|_{v} \right)^{-1} \left(\prod_{v\in S} \rho_v \right),
\end{equation}
since $a_{K} \geq \left(\sqrt{D_K}\right)^j $ and $w_I \in \mathcal O_S$ for each $I.$

\noindent Let   
\begin{align}\label{equ:rho-til}
    \displaystyle \tilde{\rho}& \defeq  \left(\sqrt{D_K}\right)^j \left( \prod_{v \in S} |\varepsilon_{v}^{(1)}|_{v} |\delta_{v}^{(1)}|_{v} \dots |\delta_{v}^{(n)}|_{v} \right)^{-1}  \left(\prod_{v\in S} \rho_v \right) \nonumber \\ & \displaystyle = \left(\sqrt{D_K}\right)^j  \left(\text{const}_K\right)^{-(n+1)} \left(\prod_{v\in S} \rho_v \right).
\end{align}

So condition (C\ref{item:C2}) holds  with 
\begin{equation}\label{equ:rho}
    \rho = \min \{1,\tilde{\rho} \}.
\end{equation}

Now as a consequence of Theorem (\ref{thm:qn}), the Proposition (\ref{main prop}) follows immediately with  
\begin{equation}\label{equ: ctil}
    \tilde{C}= C \left(N_X D^2\right)^{n+1} \left(n+1\right)^{1+\frac{\alpha |S_r|}{2}+\alpha |S_c| } \left(\frac{\sqrt{D_K}}{\rho}\right)^{\alpha} .
\end{equation}\,\,$\Box$

\end{document}